\documentclass[12pt]{article}
\usepackage{amstext}
\usepackage[]{amsmath}
\usepackage{amsfonts}
\usepackage{amssymb}
\usepackage{amsthm}
\usepackage{newlfont}
\usepackage{graphicx}
\usepackage{sectsty}
\usepackage[numbers,sort&compress]{natbib}
\usepackage[dvipsnames]{xcolor}
\newtheorem{thm}{Theorem}[section]
\newtheorem{ex}{Example}[section]
\newtheorem{cor}{Corollary}[section]
\newtheorem{lem}{Lemma}[section]
\newtheorem{prop}{Proposition}[section]

\theoremstyle{definition}
\newtheorem{defn}{Definition}[section]
\theoremstyle{remark}
\newtheorem{rem}{Remark}[section]
\numberwithin{equation}{section}

\begin{document}

\title{On conformable fractional Legendre polynomials and their convergence properties with applications}
\date{}
\author{Mahmoud Abul-Ez$^{a}$\thanks{Corresponding author: mabulez56@hotmail.com},
Mohra Zayed$^{b}$\thanks{mzayed@kku.edu.sa }, 
 \\ Ali Youssef$^{c}$\thanks{alimohammedyouseuf@yahoo.com}  and Manuel De la Sen$^{d}$\thanks{manuel.delasen@ehu.eus} \\ 
\footnotesize $^{a,c}$  Mathematics Department, Faculty of Science,
 Sohag University, Sohag 82524, Egypt.\\ 
\footnotesize $^{b}$ Mathematics Department, College of Science, 
 King Khalid University, Abha, Saudi Arabia.\\
 \footnotesize $^{d}$Institute of Research and Development of Processes, University of the Basque Country,\\
  \footnotesize48940 Leioa (Bizkaia), Spain.
}

\maketitle

\begin{abstract}
The main objective of this paper is to give a wide  study on the conformable fractional Legendre polynomials (CFLPs). This study is assumed to be a generalization and refinement, in an easy way, of the scalar case  into the context of the conformable fractional differentiation. We introduce the  CFLPs via different generating functions and provide some of their main  properties and convergence results. Subsequently, some pure recurrence and differential recurrence relations, Laplace's first integral formula and orthogonal properties are then developed for CFLPs. 
We append our study with presenting shifted CFLPs and describing applicable scheme using the collocation method to solve some fractional differential equations (FDEs) in the sense of conformable derivative. Some useful examples of FDEs are treated to support our theoretical results and examining their exact solutions. To the best of our knowledge, the obtained results are newly presented and could  enrich the fractional theory of special functions.
\end{abstract}

\noindent \textbf{Keywords}: Special functions; Legendre polynomials;  Conformable fractional calculus;
  Fractional differential equations. \vskip1mm 

\section{Introduction}
\noindent
The special functions have an increasing and recognized role in mathematical physics, fractional calculus,  theory of differential equations, quantum mechanics, approximation theory and many branches in science. They have a long history that can be traced back to the past three centuries where the problems of terrestrial and
celestial mechanics, the
boundary value problems of electromagnetism and the eigenvalue
problems of quantum mechanics had been solved.

The theory of fractional calculus is as old as classical calculus and  classified as generalized fractional integrals or derivatives.   
Nowadays, an increasing number of researchers are paying attention to the fractional
calculus, since they found that fractional order derivatives and the description of many
physical phenomena in the real world \cite{podlubny1999fractional}.
For instance, fractional derivatives give a better description of the model of 
 the nonlinear oscillation of earthquake \cite{he1998nonlinear}, also modeling the fluid-dynamics by fractional derivatives can eliminate the deficiency emerging from the occurrence of continuum flow of traffic  \cite{he1999some,moaddy2011non}. Therefore, the fractional calculus became a more convenient tool for the description of  mathematical models in different aspects of physical and dynamical  systems and so forth \cite{agarwal2010survey,kilbas2006theory,kosmatov2016resonant,lu2018time}. 
The fractional order derivative in the sense of  Riemann-Liouville or Caputo  were the main tools to achieve the forth mentioned results.

In 2014, Khalil et al. \cite{khalil2014new} introduced a new well-behaved simple  fractional derivative named the conformable fractional
derivative (CFD) by means of the usual basic limit definition of the derivative and that
break with other definitions. Then, it becomes clear as it is noted from the literature that this new definition seems to be the most convenient one in the
world of fractional calculus, since it can be considered as an elegant extension of the
classical derivative. Moreover, this new definition has developed by Abdeljawad \cite{abdeljawad2015conformable} and also, very recently, has been modified in some sense
 by El-Ajou \cite{el2020modification}.
For recent developments on conformal differentiation we refer to \cite{acan2017conformable,ammi2019existence,anderson2016results,asawasamrit2016periodic,benkhettou2016conformable,bucur2016nonlocal,hesameddini2015numerical,nwaeze2016mean,sitho2018noninstantaneous,unal2017solution}.

 The usability of the conformable derivative notion has wide areas of interest in both theoretical and practical aspects \cite{khitab2005predictive,thomas1999modelling}.
The authors in \cite{anderson2015properties,yang2018conformable,zhao2018new,zhou2018conformable}, provided some applications through partial differential equations (PDEs) in conformable sense.
 Precisely, the  Maxwell's equations have been considered in the 
  the  conformable fractional setting  to describe
electromagnetic fields of media in \cite{zhao2018new} and conformable differential equation (CDE) has been used to
the description of the sub diffusion process in \cite{zhou2018conformable}. Also, 
 some applications in quantum mechanics have been treated in the context of CFD (see for example \cite{anderson2015properties}). 
  
The ordinary differential equations (ODEs) referred to as Legendre's differential equation is common
in engineering and physics. Particularly, it occurs when solving the Laplace's equation in
 spherical coordinates. The importance of Legendre polynomials is realized in 1784 when A. Legendre was studying the attraction of spheroids and ellipsoids. 
These polynomials appear in many different areas of mathematics and physics. They may originate as solutions of the Legendre ODE such as the famous  Helmholtz's equation and analog ODEs in spherical polar 
coordinates. They arise as a consequence of demanding a complete, orthogonal sequence of functions over $\left[ -1,1 \right]$ (Gram-Schmidt orthogonalization). 
In quantum mechanics they represent angular momentum eigen functions. 

Special functions of fractional calculus appeared and received much attention due to its great applications in different disciplines of engineering and science.
As the CFD is essentially a generalized version of the first usual integer derivative,
 we show in the present paper that many known results concerning Legendre type polynomials can be translated and stated in the framework
of the CFD. 

In 2010, Saadatmandi and Dehghan in \cite{saadatmandi2010new} derived operational matrices of fractional derivatives for
 the shifted Legendre polynomials and used them to solve fractional differential equations (FDEs)
with initial boundary conditions via spectral methods. The year after, Rida and Yousef \cite{rida2011fractional}
replaced the integer order derivative in Rodrigues' formula of the Legendre polynomials with
 fractional order derivative. However, the complexity of the resulted functions made them inconvenient
 for solving some  FDEs.
 
Thereafter, in 2013, Kazem et al. \cite{kazem2013fractional} presented an orthogonal Legendre functions with fractional
 order based on the shifted Legendre polynomials, in order to get a numerical solution  for
 some FDEs. The  technique they have used was  precise and effective.
 
In 2014, Abu Hammad and Khalil \cite{hammad2014legendre} studied the Legendre conformable fractional differential
 equation. Analog to the classical context, in certain cases, it is found that some solutions  turned to be fractional polynomials.
 Further, they studied fundamental  properties of such conformal fractional
polynomials. Recently in 2020, Zayed et al. \cite{zayed2020fractional} introduced a generalized study on the shifted Legendre type polynomials of
 arbitrary fractional orders in the Caputo sense utilizing some Rodrigues formulas in the framework of matrices.
Furthermore, they gave orthogonality properties of these polynomials in some particular cases
and suggested an application to solve some kinds of FDEs.

Motivated by the above mentioned studies, we aim in this paper first to complete the work given
 by Abu Hammad and Khalill \cite{hammad2014legendre}.
  Further, we intend to extend the presented new orthogonal functions based on CFLPs to produce useful applications for solving  conformable FDEs. Our results here improve the ones given by Rajkovi\'{c} and Kiryakova \cite{rajkovic2010legendre} in the sense of conformable derivative and develop the work given by various authors in \cite{hammad2014legendre,kazem2013fractional,ccerdik2020numerical,meng2019extremal}. Besides, we obtain 
 exact analytical solutions of some FDEs in conformable sense. 

The remnant of the paper is formulated  as follows. 
We begin in section \ref{Sec2} by presenting some paramount basic concepts  of the conformable fractional calculus. 
Through section \ref{Sec3}, we define conformable fractional Legendre polynomials through generating functions in different ways.
Various forms of CFLPs in terms of hypergeometric formula and Laplace's integral form are established in section \ref{Sec4}. 
The pure recurrence relations and differential recurrence relations are the subject of section \ref{Sec5}. Conformable fractional integral formula of CFLPs is derived in section \ref{Sec6}. In section \ref{Sec7}, we introduce a detailed study on orthogonality property with applications. Section \ref{Sec8} is devoted to the study of shifted CFLPs as well as the collocation method is proposed to solve some FDEs.
 Moreover, we present some illustrative examples to justify our suggested approximation involving the collocation method. Finally, we append by general remarks and conclusions in section \ref{Sec9}.  
\section{Preliminaries}\label{Sec2}
\noindent
Nowadays  fractional calculus can be described via two trends. 
The Riemann-Liouville approach represents the first trend for which the integral operator repeatedly proceeded $n$ times and replaced it by one integral via the famous Cauchy formula where then $n!$ is changed to the Gamma function and hence the fractional integral of noninteger order is defined. 
 Evidently, the Riemann and Caputo fractional derivatives were defined by means of integrals  (see \cite{kilbas2006theory}). 
 The second trend of fractional derivative was introduced by
 Gr\"{u}nwald-Letnikove.  Their technique was based on repeating the derivative $n$ times and then fractionalizing by means of Gamma function in the binomial coefficients. 
 The characterized derivative in this calculus seemed to be complicated and is not necessarily enjoy the same fundamental properties as of  the usual case.

As we have mentioned earlier in the introduction,  the authors
in \cite{khalil2014new} introduced the definition of CFD. The definition was stated in the following way. 
For a function $f:(0,\infty )\rightarrow \mathbb{R},$ the CFD of order $\alpha$ where $0<\alpha \leq 1$ of $f(x)$ at $x>0 $ was defined by \footnote{ Unless it is otherwise stated throughout  the whole paper, the fractional number $\alpha $ takes its value such that $0< \alpha \leq 1$.}
\begin{equation}\label{sec2, eq1}
D^{\alpha }f\left( x\right) =\lim\limits_{\varepsilon \rightarrow 0} \frac{f\left( x+\varepsilon x^{1-\alpha }\right) -f\left( x\right) }{\varepsilon }
\end{equation}
and when $x=0$ we have $D^{\alpha }f(0)=\lim\limits_{x \rightarrow 0^+}D^{\alpha }f(x)$.
\begin{rem}
\noindent
\begin{itemize} 
\item[(1)] It is found in  \cite{khalil2014new} that the CFD behaves well in the product rule and the chain rule unlike the case of  the old fractional calculus where complicated formulas appear. 
\item[(2)] As an unexpected fact, the CFD of a constant function is zero  whereas the case for Riemann-Liouville fractional derivative is not.
\item[(3)] For $\alpha =1$ in  \eqref{sec2, eq1}, one gets easily the analog usual classical derivatives.
Further, note that a function can be $\alpha $-differentiable at a
point even though it is not differentiable, for instance, take $f(x)=2\sqrt{x},$ then $D^{\frac{1}{2}}f\left( x\right) =1$. Thus $D^{\frac{1}{2}}f\left( 0\right) =1$. However, $D^{1}f\left( 0\right) $ does not exist. This differs obviously from  what is known for the classical  derivatives.
\item[(4)] In order to solve the  simple fractional differential equation $D^{\frac{1}{2}}y+y=0,$ by applying the Caputo or Riemann-Liouville definitions, then it is required to use either the Laplace transform or fractional power series technique. However, using conformable definition and the fact $D^{\alpha }(e^{\frac{1}{\alpha }x^{\alpha }})=e^{\frac{1}{\alpha }x^{\alpha }},$ one can easily see that $y=ce^{-2\sqrt{x}}$ is the general solution.
\end{itemize}
\end{rem} 
In view of \eqref{sec2, eq1},  Abu Hammed and Kalil \cite{hammad2014legendre} solved the conformable fractional Legendre differential equation: 
\begin{equation}\label{sec2, eq2}
\left( 1-x^{2\alpha }\right) D^{\alpha }D^{\alpha }y-2\alpha x^{\alpha}D^{\alpha }y+\alpha ^{2}k\left( k+1\right) y=0
\end{equation}
and introduced the conformable fractional Legendre Polynomials CFLPs, $P_{\alpha k}\left( x\right) $ as its solution. In fact we proceed on to give an explicit formula of  $P_{\alpha k}\left( x\right)$ as follows.
The solution of \eqref{sec2, eq2} asserts the following coefficients (see \cite{hammad2014legendre}, Eq. (10)) 
\begin{equation*}
a_{2n}=\frac{\left( -1\right) ^{n}\left( k+2n\right) !\left[ \left( \frac{k}{2}\right) !\right] ^{2}}{k!\left( 2n\right) !\left( \frac{k}{2}+n\right)!\left( \frac{k}{2}-n\right) !}a_{0}
\end{equation*}
The constant $a_{0}$ is usually chosen so that the polynomial solution at $x=1$ equals $1$. So, the value to be given to $a_{0}$ is 
$a_{0}=\left( -1\right) ^{\frac{k}{2}}\frac{k!}{2^{k}\left[ \left( \frac{k}{2}\right) !\right] ^{2}}.$
Since \eqref{sec2, eq2} has $x=0$ as an ordinary point,  its solution will take the form 
\begin{equation}\label{10}
y=\sum\limits_{n=0}^{\infty }a_{2n}x^{2\alpha n}=\sum\limits_{n=0}^{\infty }\frac{\left( -1\right) ^{n+\frac{k}{2}}\left( k+2n\right) !}{2^{k}\left(
2n\right) !\left( \frac{k}{2}+n\right) !\left( \frac{k}{2}-n\right) !}x^{2\alpha n}
\end{equation}
Taking into account that the factorial function is always non-negative, we should have
 $\left( \frac{k}{2}-n\right) \geq 0$ and hence $n\leq \left\lfloor \frac{k}{2}\right\rfloor $ where $\left\lfloor \frac{k}{2}\right\rfloor$ is the usual floor function. Therefore \eqref{10} becomes 
\begin{equation*}
y=\sum\limits_{n=0}^{\left\lfloor \frac{k}{2}\right\rfloor }\frac{\left(-1\right) ^{n+\frac{k}{2}}\left( k+2n\right) !}{2^{k}\left( 2n\right)!\left( \frac{k}{2}+n\right)!
\left( \frac{k}{2}-n\right) !}x^{2\alpha n},
\end{equation*}
where 
$$\left\lfloor \frac{k}{2}\right\rfloor =
\begin{cases}
  \frac{k}{2} & \text{if $k$ even} \\
       
   \frac{k-1}{2} & \text{if $k$ odd}. 
\end{cases}
$$
Put $m=\frac{k}{2}-n,$  it is easy to see
\begin{equation}\label{sec2, eq3}
y=:P_{\alpha k}\left( x\right) =\sum\limits_{n=0}^{\left\lfloor \frac{k}{2}\right\rfloor }
\frac{\left( -1\right) ^{m}\left( 2k-2m\right) !}{2^{k}m!\left( k-m\right) !\left( k-2m\right) !}x^{2\alpha n};~~~ \alpha \in (0,1]
\end{equation}
which is the $\alpha k^{\text{th}}$ CFLPs. Clearly, the first four terms of $P_{\alpha k}\left( x\right)$ are: $P_0(x)=1, P_{\alpha }(x)=x^{\alpha},
 P_{2\alpha }(x)=\frac{1}{2}\left(3x^{2\alpha}-1 \right), P_{3\alpha }(x)=\frac{1}{2}\left(5x^{3\alpha}- 3x^{\alpha} \right). $

 A common tool to be used in much of the work presented here is the rearrangement of terms in iterated series. The following two fundamental lemmas are of the kind needed to simplify many proofs in later work. For the infinite double series we have the following Lemmas (see \cite{rainville1969special}). 
 \begin{lem}\label{Lem1new}
 
 \begin{align}
 \sum_{n=0}^{\infty}\sum_{k=0}^{\infty}a_{k,n}&=\sum_{m=0}^{\infty}\sum_{j=0}^{m}a_{j,m-j}
 =\sum_{n=0}^{\infty}\sum_{k=0}^{n}a_{k,n-k}\\
 \sum_{n=0}^{\infty}\sum_{k=0}^{n}b_{k,n}&=\sum_{n=0}^{\infty}\sum_{k=0}^{\infty}b_{k,n+k}
 \end{align}
 \end{lem}

 \begin{lem}\label{Lem2new}
 
 \begin{align}
 \sum_{n=0}^{\infty}\sum_{k=0}^{\infty}a_{k,n}&=\sum_{n=0}^{\infty}\sum_{k=0}^{\lfloor \frac{n}{2}  \rfloor}a_{k,n-2k}
 \\
 \sum_{n=0}^{\infty}\sum_{k=0}^{\lfloor \frac{n}{2}  \rfloor}b_{k,n}&=\sum_{n=0}^{\infty}\sum_{k=0}^{\infty}b_{k,n+2k}\label{2Lema2}
 \end{align}
 \end{lem} 
 A combination of lemmas \ref{Lem1new} and \ref{Lem2new} gives
 
 \begin{lem}\label{Lem3new}
 \begin{equation}
 \sum_{n=0}^{\infty}\sum_{k=0}^{n}c_{k,n}=\sum_{n=0}^{\infty}\sum_{k=0}^{\lfloor \frac{n}{2}  \rfloor}c_{k,n-k}
 \end{equation}
 \end{lem}

\begin{rem}\label{Remark*}
The orthogonal property of Legendre polynomials is so important in many physical applications on the interval $[-1,1]$. So that the authors in \cite{hammad2014legendre}
studied the orthogonality of the conformable fractional Legendre polynomials CFLPs on such interval. 
For this purpose the authors extended the definition of CFD \eqref{sec2, eq1} to include the
negative values of $x,$ by assuming $\alpha $  to be of the form $\frac{1}{k},$ with $k$ an odd natural number, $k=2j+1$  for $j$  any natural number.
In such a case $x^{1-\alpha }$ will be defined for all $x\in \mathbb{R},$ and then $x^{\alpha n}$ is defined for all $x\in \mathbb{R}$ and all $n.$ Therefore, the authors in \cite{hammad2014legendre} gave the extended definition of the CFD as follows.
\end{rem}

\begin{defn}\label{def 2.1}
For $0<\alpha <\frac{1}{2j+1}\leq 1$ and  $j\in \mathbb{N}$, the $\alpha$-derivative of $f:(0,\infty)\rightarrow \mathbb{R}$ is  $$D^{\alpha }f\left( x\right)=\lim\limits_{\varepsilon \rightarrow 0}\frac{f\left( x+\varepsilon
x^{1-\alpha }\right) -f\left( x\right) }{\varepsilon },x\neq 0.$$ If $x=0,$ then $D^{\alpha }f\left( 0\right) =\lim\limits_{x\rightarrow 0^{+}}D^{\alpha}f\left( x\right)$ provided the limits exist.
\end{defn}
Consequently this definition shows that for the case of fractional polynomials one  should have \ $D^{\alpha }x^{\alpha n}=\alpha nx^{\alpha \left( n-1\right) },$ for all $x\in \mathbb{R}.$  
 
Moreover, for the $\alpha$-fractional integral of a function $f$,
the authors in \cite{khalil2014new}  suggested the following definition. 

\begin{defn}\label{DefFI}
Suppose that $f:(0,\infty)\rightarrow \mathbb{R}$ is $\alpha$-differentiable, $\alpha \in (0,1]$, then  the  $\alpha$-fractional integral of $f$  is defined by
\begin{equation*}
I_{\alpha }^{a}f\left(t \right) = I_{1 }^{a}\left( t^{\alpha -1}f \right)=
  \int\limits_{a}^{t} \frac{f\left( x\right) }{x^{1-\alpha }}dx,\ t\geq 0.
\end{equation*}
\end{defn}

In that context, they also showed that $D^{\alpha}\left( I^{\alpha}_{a}(f) \right)(t)=f(t)$. According to the discussion in Remark \ref{Remark*},  let $I_{\alpha }^{-1}\left( f\right) \left( 1\right) =\int\limits_{-1}^{1} \frac{f\left( x\right) }{x^{1-\alpha }}dx,$  from which  the orthogonality
   property  of $P_{\alpha n} (x)$ for $n\neq m$ was given as  \cite{hammad2014legendre}
\begin{equation}\label{sec2, eq4}
\int\limits_{-1}^{1}P_{\alpha n}\left( x\right) P_{\alpha m}\left( x\right)
x^{1-\alpha }dx=0,~~~m\neq n.
\end{equation}

In the sequel we are going to develop this orthogonality property for the case $n=m$ in order to be usable in giving some applications. 
Motivated by the above discussions and along with the work given in literature concerning CFLPs, further investigations involving CFLPs will be explored throughout the present work. We start with defining CFLPs via different generating functions as follows.
 

\section{Conformable  fractional Legendre polynomials within generating functions}\label{Sec3}
\noindent
Along with the scalar case, one can give a generating function to define the CFLPs, $P_{\alpha n}\left( x\right) $ through the following result as follows.

\begin{thm}\label{sec3, thm1}
For  $\alpha \in (0,1],$  the generating relation of CFLPs, $P_{\alpha n}\left( x\right) $ can be given by the following formula 
\begin{equation}\label{sec3, eq1}
g(x,t)=\frac{1}{\sqrt{1-2x^{\alpha }t^{\alpha }+t^{2\alpha }}}=\sum\limits_{n=0}^{\infty }P_{\alpha n}\left( x\right) t^{\alpha n}.
\end{equation}
\end{thm}

\begin{proof}
 The function $\left( 1-2x^{\alpha }t^{\alpha }+t^{2\alpha}\right) ^{\frac{-1}{2}}$ can be formed by means of the Gauss hypergeometric function as
\begin{equation*}
\left( 1-2x^{\alpha }t^{\alpha }+t^{2\alpha }\right) ^{\frac{-1}{2}}= ~_{1}F_{0}\left( \frac{1}{2};-;2x^{\alpha }t^{\alpha }-t^{2\alpha}\right).
\end{equation*}
Therefore,
\begin{eqnarray*}
\begin{aligned}
\left( 1-2x^{\alpha }t^{\alpha }+t^{2\alpha }\right) ^{\frac{-1}{2}}&= ~_{1}F_{0}\left( \frac{1}{2};-;2x^{\alpha }t^{\alpha }-t^{2\alpha}\right)\\
&=\sum\limits_{n=0}^{\infty }\frac{\left( \frac{1}{2}\right)
_{n}\left( 2x^{\alpha }t^{\alpha }-t^{2\alpha }\right)
^{n}}{n!}\\
&=\sum\limits_{n=0}^{\infty }\sum\limits_{k=0}^{n}\frac{\left( \frac{1}{2}
\right) _{n}}{~n!}\frac{n!\left( -1\right) ^{k}2^{n-k}}{\left( n-k\right) !k!}x^{\alpha \left(
n-k\right) }t^{\alpha \left( n+k\right) }
\end{aligned}
\end{eqnarray*}
With the help of lemma \ref{Lem3new}, we infer that
\begin{equation*}
\begin{split}
\left( 1-2x^{\alpha }t^{\alpha }+t^{2\alpha }\right) ^{\frac{-1}{2}}&=~\sum\limits_{n=0}^{\infty }\sum\limits_{k=0}^{\left\lfloor\frac{n}{2}
\right\rfloor }\frac{\left( -1\right) ^{k}\left( \frac{1}{2}\right)
_{n-k}2^{n-2k}}{\left( n-2k\right) !k!}x^{\alpha \left( n-2k\right) }  t^{\alpha n}\\
&=\sum\limits_{n=0}^{\infty }\sum\limits_{k=0}^{\left\lfloor \frac{n}{2}
\right\rfloor }\frac{\left( -1\right) ^{k}\left( 2n-2k\right) !}{2^{n}(n-k)!\left( n-2k\right) !k!} x^{\alpha \left( n-2k\right) }  t^{\alpha n} 
=\sum\limits_{n=0}^{\infty }P_{n\alpha }\left( x\right) t^{n\alpha}
\end{split}
\end{equation*}
and the result follows.
\end{proof}

\subsection{Further generating functions}
The previously mentioned expression  $g(x,t)= \left( 1-2x^{\alpha }t^{\alpha }+t^{2\alpha}\right) ^{\frac{-1}{2}}, $  used to define CFLPs,
can be rewritten in different ways as an expansion in powers of $t$ to yield additional results. Consider
\begin{equation*}
\begin{split}
\left(1-2x^{\alpha }t^{\alpha }+t^{2\alpha }\right) ^{\frac{-1}{2}}&=\left(1-x^{\alpha }t^{\alpha }\right) ^{-1}\left( 1-\frac{t^{2\alpha }
\left(x^{2\alpha }-1\right)}{\left( 1-x^{\alpha }t^{\alpha }\right) ^{2}}\right) ^{\frac{-1}{2}}\\
&=\left( 1-x^{\alpha }t^{\alpha }\right) ^{-1}\
_{1}F_{0}\left( \frac{1}{2};-;\frac{t^{2\alpha }\left( x^{2\alpha }-1\right) }{
\left( 1-x^{\alpha }t^{\alpha }\right) ^{2}}\right) \\
&=\left( 1-x^{\alpha }t^{\alpha }\right) ^{-1}\ \sum\limits_{k=0}^{\infty }
\frac{\left( \frac{1}{2}\right) _{k}}{~k!}\frac{t^{2k\alpha }\left(
x^{2\alpha }-1\right) ^{k}}{\left( 1-x^{\alpha }t^{\alpha }\right) ^{2k}} \\
&=\sum\limits_{k=0}^{\infty }\frac{\left( \frac{1
}{2}\right) _{k}~\ t^{2k\alpha }\left( x^{2\alpha }-1\right) ^{k}}{k!}
~_{1}F_{0}\left( 2k+1;-;x^{\alpha }t^{\alpha }\right) \\
&=\sum\limits_{n=0}^{\infty}\sum\limits_{k=0}^{\infty }\frac{\left( \frac{1}{2}\right) _{k}~\left( n+2k\right) !\ 
\left( x^{2\alpha }-1\right)
^{k}x^{n\alpha }}{\left(2k\right) !k!n!}t^{\alpha \left( n+2k\right) }
\end{split}
\end{equation*}
Using lemma \ref{Lem2new}, we can see that 
\begin{eqnarray}\label{sec3, eq3}
\left( 1-2x^{\alpha }t^{\alpha }+t^{2\alpha }\right) ^{\frac{-1}{2}
}=\sum\limits_{n=0}^{\infty }\sum\limits_{k=0}^{\left\lfloor \frac{n}{2}
\right\rfloor }\frac{\left( \frac{1}{2}\right) _{k}~n!\ \left( x^{2\alpha
}-1\right) ^{k}x^{\alpha \left( n-2k\right) }}{\left( 2k\right) !k!\left(
n-2k\right) !}t^{\alpha n}
\end{eqnarray}
Thus it can be verified from the formula \eqref{sec3, eq1}  that
\begin{equation}\label{sec3, eq4}
P_{\alpha n}\left( x\right) =\sum\limits_{k=0}^{\left\lfloor \frac{n}{2}%
\right\rfloor }\frac{\left( \frac{1}{2}\right) _{k}~n!\ }{\left( 2k\right)
!k!\left( n-2k\right) !}x^{\alpha \left( n-2k\right) }\left( x^{2\alpha
}-1\right) ^{k}
\end{equation}
\begin{rem} (Additional useful formula of $P_{\alpha n}\left(x\right) $) \newline
Again, reformulation of the generating function $g(x,t)$ given in \eqref{sec3, eq1} yields

\begin{equation*}
\begin{split}
\left( 1-2x^{\alpha }t^{\alpha }+t^{2\alpha }\right) ^{\frac{-1}{2}
}&=\left( 1+t^{\alpha }\right)^{-1}\left( 1-\frac{2t^{\alpha }\left(
x^{\alpha }+1\right) } {\left(1+t^{\alpha }\right) ^{2}}\right) ^{\frac{-1}{2
}}\\
&=\left( 1+t^{\alpha}\right) ^{-1}~_{1}F_{0}\left( \frac{1}{2};-;\frac{
2t^{\alpha }\left(x^{\alpha }+1\right) }{\left( 1+t^{\alpha }\right) ^{2}}
\right) \\
&=\sum\limits_{k=0}^{
\infty }\frac{\left( \frac{1}{2}\right) _{k}}{~k!}\frac{2^{k}t^{\alpha k}
\left(x^{\alpha }+1\right) ^{k}}{\left( 1+t^{\alpha }\right) ^{2k+1}} \\
&=\sum\limits_{k=0}^{\infty }\frac{\left( \frac{1}{2}\right)
_{k}2^{k}t^{\alpha k}\left( x^{\alpha }+1\right) ^{k}}{~k!}~_{1}F_{0}\left(
2k+1;-;-t^{\alpha }\right) \\
&=\sum\limits_{n=0}^{\infty }\sum\limits_{k=0}^{\infty }\frac{\left( \frac{1
}{2}\right) _{k}~\left( -1\right) ^{n}\ 2^{k}~\left( 2k+1\right) _{n}~\left(
x^{\alpha }+1\right) ^{k}}{~k!n!}t^{\alpha \left( n+k\right) } \\
&=\sum\limits_{n=0}^{\infty } \sum\limits_{k=0}^{\infty }\frac{~\ ~\left(
-1\right) ^{n}\left( n+2k\right) !~\left( x^{\alpha }+1\right) ^{k}}{
2^{k}~\left( k!\right) ^{2}n!}t^{\alpha \left( n+k\right) }
\end{split}
\end{equation*}
The use of  lemma \ref{Lem1new} implying that  
\begin{equation*}
\left( 1-2x^{\alpha }t^{\alpha }+t^{2\alpha }\right) ^{\frac{-1}{2}
}=\sum\limits_{n=0}^{\infty }\bigskip \sum\limits_{k=0}^{n }\frac{
~\left( -1\right) ^{n+k}\left( n+k\right) !~\left( x^{\alpha }+1\right) ^{k}
}{2^{k}~\left( k!\right) ^{2}\left( n-k\right) !}t^{\alpha n}
\end{equation*}
Thus according to the generating function \eqref{sec3, eq1}, it follows that
\begin{equation}\label{sec3, eq8}
P_{n\alpha }\left( x\right)  =\sum\limits_{k=0}^{n}\frac{
\left( -1\right) ^{n+k}\left( n+k\right) }{ \left( k!\right) ^{2}\left(
n-k\right) !}\left( \frac{x^{\alpha }+1}{2} \right) ^{k}
\end{equation}
\end{rem}
We append this section by the following interesting result.
\begin{thm}\label{sec3, thm2}
For $\alpha \in (0,1],$ the CFLPs
can be written through the hypergeometric function in the form
\begin{equation}\label{sec3, eq5}
\left( 1-x^{\alpha }t^{\alpha }\right) ^{-c}~_{2}F_{1}\left( \frac{1}{2}c,
\frac{1}{2}c+\frac{1}{2};1;\frac{\left( x^{2\alpha }-1\right) t^{2\alpha }}{
\left( 1-x^{\alpha }t^{\alpha }\right) ^{2}}\right)
=\sum\limits_{n=0}^{\infty }\frac{\left( c\right) _{n}P_{\alpha n}\left(
x\right) }{n!}t^{\alpha n}, 
\end{equation}
where $c$ is an arbitrary real number.
\end{thm}
\begin{proof}
Employing \eqref{sec3, eq4} one may get a new generating function for CFLPs, $P_{\alpha n}\left( x\right) $. Then, for arbitrary real number $c$, we have 
\begin{equation}\label{sec3, eq6}
\sum\limits_{n=0}^{\infty }\frac{\left( c\right) _{n}P_{\alpha n}\left(
x\right) }{n!}t^{\alpha n}=\sum\limits_{n=0}^{\infty
}\sum\limits_{k=0}^{\left\lfloor \frac{n}{2}\right\rfloor }\frac{\left(
c\right) _{n}\left( \frac{1}{2}\right) _{k}~x^{\alpha \left( n-2k\right)
}\left( x^{2\alpha }-1\right) ^{k}}{\left( 2k\right) !k!\left( n-2k\right) !}
t^{\alpha n}
\end{equation}
Owing to the  relation \eqref{2Lema2}  and lemma \ref{Lem2new}, we rewrite  \eqref{sec3, eq6}   in the form
\begin{equation}\label{sec3, eq7}
\sum\limits_{n=0}^{\infty }\frac{\left( c\right) _{n}P_{\alpha n}\left(
x\right) }{n!}t^{\alpha n}=\sum\limits_{n=0}^{\infty
}\sum\limits_{k=0}^{\infty }\frac{\left( c\right) _{n+2k}\left( \frac{1}{2}
\right) _{k}~x^{\alpha n}\left( x^{2\alpha }-1\right) ^{k}}{\left( 2k\right)
!k!n!}t^{\alpha \left( n+2k\right) }
\end{equation}
But Pochhammer symbol properties give
\begin{equation*}
\begin{split}
\left( c\right) _{n+2k}=\left( c+2k\right) _{n}\left( c\right) _{2k}&=\left(
c+2k\right) _{n}\left( \frac{1}{2}c\right) _{k}\left( \frac{1}{2}c+\frac{1}{2
}\right) _{k}2^{2k}\\
&=\left( c+2k\right) _{n}\left( \frac{1}{2}c\right)
_{k}\left( \frac{1}{2}c+\frac{1}{2}\right) _{k}\frac{\left( 2k\right) !}{
k!\left( \frac{1}{2}\right)_{k}}.
\end{split}
\end{equation*}
Therefore, in view of \eqref{sec3, eq7}, we get
\begin{equation*}
\begin{split}
\sum\limits_{n=0}^{\infty }\frac{\left( c\right) _{n}P_{\alpha n}\left(
x\right) }{n!}t^{\alpha n} &=\sum\limits_{k=0}^{\infty}\sum\limits_{n=0}^{\infty }\frac{\left( c+2k\right) _{n}\left( x^{\alpha}t^{\alpha }\right) 
^{n}}{n!}\frac{~\left( \frac{1}{2}c\right) _{k}\left( \frac{1}{2}c+\frac{1}{2}\right) _{k}\left( x^{2\alpha }-1\right) ^{k}}{\left( k!\right) ^{2}}t^{\alpha \left( 2k\right) } \\
&=\sum\limits_{k=0}^{\infty }\ _{1}F_{0}\left( c+2k;-;x^{\alpha }t^{\alpha}\right) \frac{~\left( \frac{1}{2}c\right) _{k}\left( \frac{1}{2}c+\frac{1}{2}\right) _{k}
\left( x^{2\alpha }-1\right) ^{k}}{\left( k!\right) ^{2}}t^{2\alpha k} \\
&=\left( 1-x^{\alpha }t^{\alpha }\right) ^{-c}\sum\limits_{k=0}^{\infty }\ 
\frac{~\left( \frac{1}{2}c\right) _{k}\left( \frac{1}{2}c+\frac{1}{2}\right)
_{k}\left( x^{2\alpha }-1\right) ^{k}t^{2\alpha k}}{\left( k!\right)
^{2}\left( 1-x^{\alpha }t^{\alpha }\right) ^{2k}} \\
&=\left( 1-x^{\alpha }t^{\alpha }\right) ^{-c}~_{2}F_{1}\left( \frac{1}{2}c,
\frac{1}{2}c+\frac{1}{2};1;\frac{\left( x^{2\alpha }-1\right) t^{2\alpha }}{
\left( 1-x^{\alpha }t^{\alpha }\right) ^{2}}\right) 
\end{split}
\end{equation*}
\end{proof}
\begin{rem}
In the previous theorem \ref{sec3, thm2}, as  $c$  be any real number, thus for $c=1,$  \eqref{sec3, eq5}
generates into   \eqref{sec3, eq1}. If $c=0$ or $c$ is a negative integer, both sides of  \eqref{sec3, eq5}  terminate, and only a finite number of terms  of
CFLPs is then generated by  \eqref{sec3, eq5}.  
\end{rem}
\section{Hypergeometric and integral forms of CFLPs}\label{Sec4}
In this section we introduce the CFLPs, $P_{\alpha n}\left(x\right) $ in terms of the Gauss hypergeometric form. Three different formula’s  are obtained via the next given results. Moreover, we establish Laplace's first integral form of CFLPs. 
\begin{thm}
For $\alpha \in (0,1],$ the CFLPs,  $P_{\alpha n}\left(x\right) $ can be defined  by means of Gauss hypergeometric function as 
\begin{equation}\label{sec4, eq1}
P_{\alpha n}\left( x\right) =~_{2}F_{1}\left( -n,n+1;1;\frac{1-x^{\alpha }}{2}\right) 
\end{equation}
\end{thm}
\begin{proof}
 First we have 
\begin{eqnarray*}
\begin{aligned}
 \left( 1-2x^{\alpha }t^{\alpha }+t^{2\alpha }\right) ^{\frac{-1}{2}
}&=\left( 1-t^{\alpha
}\right) ^{-1}\left( 1+\frac{2t^{\alpha }\left( x^{\alpha }-1\right) }{
\left( 1-t^{\alpha }\right) ^{2}}\right) ^{\frac{-1}{2}}\\
&=\left( 1-t^{\alpha
}\right) ^{-1}~_{1}F_{0}\left( \frac{1}{2};-;\frac{2t^{\alpha }\left(
x^{\alpha }-1\right) }{\left( 1-t^{\alpha }\right) ^{2}}\right)\\
&=\sum\limits_{k=0}^{\infty }\frac{\left( \frac{1}{2}\right) _{k}}{~k!}\frac{
2^{k}t^{\alpha k}\left( x^{\alpha }-1\right) ^{k}}{\left( 1-t^{\alpha
}\right) ^{2k+1}}\\
&=\sum\limits_{k=0}^{\infty }\frac{\left( \frac{1}{2}\right)
_{k}2^{k}t^{\alpha k}\left( x^{\alpha }-1\right) ^{k}}{~k!}~_{1}F_{0}\left(2k+1;-;t^{\alpha }\right)\\
&=\sum\limits_{n=0}^{\infty }\sum\limits_{k=0}^{\infty }\frac{~\
~\left( n+2k\right) !~\left( x^{\alpha }-1\right) ^{k}}{2^{k}~\left(k!\right) ^{2}n!}t^{\alpha \left( n+k\right) }
\end{aligned}
\end{eqnarray*}
Relying to lemma \ref{Lem1new}, one obtains  
\begin{equation}\label{Eq 4.2}
\left( 1-2x^{\alpha }t^{\alpha }+t^{2\alpha }\right) ^{\frac{-1}{2}
}=\sum\limits_{n=0}^{\infty } \sum\limits_{k=0}^{n }\frac{
\left( n+k\right) !~\left( x^{\alpha }-1\right) ^{k}}{2^{k}~\left(
k!\right) ^{2}\left( n-k\right) !}t^{\alpha n}
\end{equation}
Note that the right hand side equal zero if $k>n$, and then \eqref{Eq 4.2} becomes

\begin{equation*}
\left( 1-2x^{\alpha }t^{\alpha }+t^{2\alpha }\right) ^{\frac{-1}{2}
}=\sum\limits_{n=0}^{\infty } \sum\limits_{k=0}^{\infty }\frac{
\left( n+k\right) !~\left( x^{\alpha }-1\right) ^{k}}{2^{k}~\left(
k!\right) ^{2}\left( n-k\right) !}t^{\alpha n}
\end{equation*}
Owing to \eqref{sec3, eq1} we have
\begin{equation}\label{sec4, eq2}
P_{\alpha n}\left( x\right) =\sum\limits_{k=0}^{\infty }\frac{~\
~\left( n+k\right) !~}{\left( k!\right) ^{2}\left( n-k\right) !}\left( \frac{
x^{\alpha }-1}{2}\right) ^{k}
\end{equation}
Since $\frac{\left( -n\right) _{k}\left( n+1\right) _{k}}{\left( 1\right) _{k}}
=\frac{\left( -1\right) ^{k}\left( n+k\right) !}{\left( n-k\right) !k!},$ then \eqref{sec4, eq2} gives 
\begin{eqnarray*}
P_{\alpha n}\left( x\right) &=\sum\limits_{k=0}^{\infty
}\frac{\left( -n\right) _{k}\left( n+1\right) _{k}}{\left( 1\right) _{k}~k!}
\left( \frac{1-x^{\alpha }}{2}\right) ^{k}
=~_{2}F_{1}\left( -n,n+1;1;\frac{
1-x^{\alpha }}{2}\right) 
\end{eqnarray*}
as required.
\end{proof}

\begin{thm}
For $\alpha \in (0,1],$  the CFLPs, $P_{\alpha n}\left(x\right) $ can be formulated by means of Gauss hypergeometric function as 
\begin{equation}\label{sec4, eq3}
P_{\alpha n}\left( x\right) =\frac{\left( \frac{1}{2}\right) _{n}\left(2x^{\alpha }\right) 
^{n}}{n!}~_{2}F_{1}\left( \frac{-1}{2}n,\frac{-1}{2}n+\frac{1}{2};\frac{1}{2}-n;\frac{1}{x^{2\alpha }}\right). 
\end{equation}
\end{thm}
\begin{proof}
In view of \eqref{sec2, eq3}, we have
\begin{equation}\label{sec4, eq4}
P_{\alpha n}\left( x\right) =\sum\limits_{k=0}^{\left\lfloor \frac{n}{2}\right\rfloor }\frac{\left( -1\right) ^{k}\left( 2n-2k\right) !}{2^{n}k!\left( n-k\right)!\left( n-2k\right) !}x^{\alpha \left( n-2k\right) }
\end{equation}
Since 
\begin{equation*}
\frac{\left( \frac{1}{2}\right) _{n}\left( -n\right)_{2k}}{\left( \frac{1}{2}-n\right) _{k}}
=\frac{\left( -1\right)^{k}2^{2k}n!\left( 2n-2\right)!}{2^{2n}\left( n-k\right)!\left(n-2k\right)!},
\end{equation*}
then it follows that
\begin{equation}\label{sec4, eq5}
\begin{split}
P_{\alpha n}\left( x\right) &=\sum\limits_{k=0}^{\left\lfloor \frac{n}{2}\right\rfloor }\frac{\left( \frac{1}{2}\right) _{n}
\left( -n\right) _{2k}}{n!k!\left( \frac{1}{2}-n\right) _{k}}\left( 2x^{\alpha }\right) ^{\left(n-2k\right) }\\
&=\frac{\left( \frac{1}{2}\right) _{n}2^{n}x^{\alpha n}}{n!}
\sum\limits_{k=0}^{\left\lfloor \frac{n}{2}\right\rfloor }\frac{\left(-n\right) _{2k}}{k!\left( \frac{1}{2}-n\right) _{k}}
\left( 2x^{\alpha}\right) ^{-2k}
\end{split}
\end{equation}
Note that $\left( -n\right) _{2k}=0$ if $k>\left\lfloor \frac{n}{2}\right\rfloor$, and then \eqref{sec4, eq5} becomes
\begin{equation}\label{sec4, eq6}
P_{\alpha n}\left( x\right) =\frac{\left( \frac{1}{2}\right)
_{n}2^{n}x^{\alpha n}}{n!}\sum\limits_{k=0}^{\infty }\frac{\left( -n\right)
_{2k}}{k!\left( \frac{1}{2}-n\right) _{k}}\left( 2x^{\alpha }\right) ^{-2k}
\end{equation}
Since $\left( -n\right) _{2k}=2^{2k}$ $\left( \frac{-1}{2}n\right) _{k} \left(\frac{-1}{2}n+\frac{1}{2}\right) _{k}$ holds for Pochhammer symbol, then we can deduce

\begin{equation*}
\begin{split}
P_{\alpha n}\left( x\right) &=\frac{\left( \frac{1}{2}\right)
_{n}2^{n}x^{\alpha n}}{n!}\sum\limits_{k=0}^{\infty }\frac{\left( \frac{-1}{2}n\right) _{k}\left( \frac{-1}{2}n+\frac{1}{2}\right) _{k}}{k!\left( \frac{1}{2}-n\right) _{k}}\left( x^{-2\alpha }\right) ^{k}\\
&=\frac{\left( \frac{1}{2}\right) _{n}\left( 2x^{\alpha }\right) ^{n}}{n!}~_{2}F_{1}\left( \frac{-1}{2}n,\frac{-1}{2}n+\frac{1}{2};\frac{1}{2}-n;\frac{1}{x^{2\alpha }}\right), 
\end{split}
\end{equation*}
which finishes the proof.
\end{proof}

\begin{thm}
For $\alpha \in (0,1],$ we can write the CFLPs, $P_{\alpha n}\left( x\right) $ by means of hypergeometric functions as:
\begin{equation}\label{sec4, eq7}
P_{\alpha n}\left( x\right) =x^{\alpha n}~_{2}F_{1}\left( \frac{-1}{2}n,\frac{-1}{2}n+\frac{1}{2};1;\frac{x^{2\alpha }-1}{x^{2\alpha }}\right) 
\end{equation}
\end{thm}
\begin{proof}
Relation \eqref{sec3, eq4} gives
\begin{equation}\label{sec4, eq8}
\begin{split}
P_{\alpha n}\left( x\right) &=\sum\limits_{k=0}^{\left\lfloor \frac{n}{2}
\right\rfloor }\frac{n!}{2^{2k}\left( k!\right) ^{2}\left( n-2k\right) !}
x^{\alpha \left( n-2k\right) }.\left( x^{2\alpha }-1\right)
^{k}\\
&=\sum\limits_{k=0}^{\left\lfloor \frac{n}{2}\right\rfloor }\frac{\left(
-n\right) _{2k}x^{\alpha n}}{2^{2k}\left( k!\right) ^{2}}.\left( \frac{
x^{2\alpha }-1}{x^{2\alpha }}\right) ^{k}
\end{split}
\end{equation} 
But $\left( -n\right) _{2k}=0$ if $k>\left\lfloor \frac{n}{2}\right\rfloor$,  then using \eqref{sec4, eq8} it follows that
\begin{equation}\label{sec4, eq9}
P_{\alpha n}\left( x\right) =\sum\limits_{k=0}^{\infty }\frac{\left(
-n\right) _{2k}x^{\alpha n}}{2^{2k}\left( k!\right) ^{2}}.\left( \frac{
x^{2\alpha }-1}{x^{2\alpha }}\right) ^{k}
\end{equation}
Since $\left( -n\right) _{2k}=2^{2k}$ $\left(\frac{-1}{2}n\right) _{k} \left( \frac{-1}{2}n+\frac{1}{2}\right) _{k},$ then
\begin{equation*}
\begin{split}
P_{\alpha n}\left( x\right) &=x^{\alpha n}\sum\limits_{k=0}^{\infty }\frac{
\left( \frac{-1}{2}n\right) _{k}\left( \frac{-1}{2}n+\frac{1}{2}\right) _{k}
}{\left( k!\right) ^{2}}.\left( \frac{x^{2\alpha }-1}{x^{2\alpha }}\right)
^{k}\\
&=x^{\alpha n}~_{2}F_{1}\left( \frac{-1}{2}n,\frac{-1}{2}n+\frac{1}{2};1;
\frac{x^{2\alpha }-1}{x^{2\alpha }}\right). 
\end{split}
\end{equation*}
The relation \eqref{sec4, eq7} is therefore attained as required. 
\end{proof}


\begin{thm} (Laplace's first integral form of CFLPs) \newline
For $\alpha \in (0,1],$  the CFLPs, $P_{\alpha n}\left( x\right) $ can be
represented via the integral form as
\begin{equation}\label{sec5, eq1}
P_{\alpha n}\left( x\right) =\frac{1}{\pi }\int\limits_{0}^{\pi }\left(x^{\alpha }+\sqrt{x^{2\alpha }-1}\cos \varphi \right) ^{n}d\varphi. 
\end{equation}
\end{thm}

\begin{proof}
The integral can be recognized as the usual Beta function $\beta (x,y)$, which has the following property 
\begin{eqnarray*}
\frac{\left( \frac{1}{2}\right) _{k}}{k!}=\frac{1}{\pi }\beta \left( \frac{1}{2},k+\frac{1}{2}\right)
=\frac{1}{\pi }.\int\limits_{0}^{\pi }\cos^{2k}\varphi ~d\varphi 
\end{eqnarray*}
Therefore, in view of \eqref{sec3, eq4}, one easily gets 
\begin{equation}\label{sec5, eq2}
P_{\alpha n}\left( x\right) =\frac{1}{\pi }\sum\limits_{k=0}^{\left\lfloor
\frac{n}{2}\right\rfloor }\frac{~n!\ }{\left( 2k\right) !\left( n-2k\right)!}
x^{\alpha \left( n-2k\right) }\left( x^{2\alpha }-1\right)
^{k}.\int\limits_{0}^{\pi }\cos ^{2k}\varphi ~d\varphi 
\end{equation}
Note that for $m$ is odd we see $\int\limits_{0}^{\pi }\cos ^{m}\varphi ~d\varphi =0,$  and then we obtain
\begin{eqnarray*}
\begin{aligned}
P_{\alpha n}\left( x\right) &=\frac{1}{\pi }\int\limits_{0}^{\pi
}\sum\limits_{k=0}^{n}\frac{~n!\ }{\left( k\right) !\left( n-k\right)!}
x^{\alpha \left( n-k\right) }\left( x^{2\alpha }-1\right) ^{\frac{k}{2}}.\cos ^{k}\varphi ~d\varphi \\
&=\frac{1}{\pi }\int\limits_{0}^{\pi }\left( x^{\alpha }+
\sqrt{x^{2\alpha }-1}\cos \varphi \right) ^{n}d\varphi 
\end{aligned}
\end{eqnarray*}
as desired.
\end{proof}

\section{The pure recurrence relation and differential recurrence relations}\label{Sec5}
In this section, we first give the pure recurrence relation which enable us to obtain  some of differential recurrence relations. We start with the following result.
\begin{thm}
The conformable fractional Legendre function $P_{\alpha n}\left( x\right) $ satisfies the following pure recurrence relation
\begin{equation}\label{sec6, eq1}
P_{\alpha \left( n+1\right) }\left( x\right) =\frac{\left( 2n+1\right) }{
\left( n+1\right) }~x^{\alpha }~P_{\alpha n}\left( x\right) -\frac{n}{n+1}
P_{\alpha \left( n-1\right) }\left( x\right) 
\end{equation}
\end{thm}
\begin{proof}
Acting by fractional differential operator $D^{\alpha }$ on \eqref{sec3, eq1} with respect to $t$, we get 
\begin{equation*}
\frac{-1}{2}~\left( 1-2x^{\alpha }t^{\alpha }+t^{2\alpha }\right) ^{\frac{-3}{2}}
\left( -2\alpha x^{\alpha }t^{\alpha -1}+2\alpha t^{2\alpha -1}\right).t^{1-\alpha }=\sum\limits_{n=0}^
{\infty }n\alpha P_{\alpha n}\left(x\right)  t^{\alpha \left( n-1\right)}
\end{equation*}
Note that
\begin{equation*}
\alpha \left( x^{\alpha }-t^{\alpha }\right) \left( 1-2x^{\alpha }t^{\alpha}+t^{2\alpha }\right) 
^{\frac{-3}{2}}=\alpha \sum\limits_{n=0}^{\infty}nP_{\alpha n}\left( x\right) t^{\alpha\left( n-1\right)}.
\end{equation*} 
Thus
\begin{equation}\label{(5.2)}
\left( x^{\alpha }-t^{\alpha }\right) \left( 1-2x^{\alpha }t^{\alpha}
+t^{2\alpha }\right) ^{\frac{-1}{2}}=\left( 1-2x^{\alpha }t^{\alpha}
+t^{2\alpha }\right) \sum\limits_{n=0}^{\infty }nP_{\alpha n}\left(x\right) t^{\alpha \left( n-1\right)}
\end{equation}
Inserting \eqref{(5.2)} in  \eqref{sec3, eq1} we obtain 
\begin{equation*}
\left( x^{\alpha }-t^{\alpha }\right) \sum\limits_{n=0}^{\infty }P_{\alpha n}
\left( x\right) \ t^{\alpha n}=\left( 1-2x^{\alpha }t^{\alpha }+t^{2\alpha}\right) 
\sum\limits_{n=0}^{\infty }nP_{\alpha n}\left( x\right) t^{\alpha \left( n-1\right)}
\end{equation*}
Then we arrive to
\begin{eqnarray}\label{May}
\begin{aligned}
\sum\limits_{n=0}^{\infty }x^{\alpha }P_{\alpha n}\left( x\right) \
t^{\alpha n}-\sum\limits_{n=0}^{\infty }P_{\alpha n}\left( x\right) \
t^{ \alpha \left( n+1\right)}=&\sum\limits_{n=0}^{\infty }nP_{\alpha n}\left(
x\right) t^{ \alpha \left( n-1\right)}-\sum\limits_{n=0}^{\infty}2nx^{\alpha }P_{\alpha n}\left( x\right) t^{\alpha n}\\
&+\sum\limits_{n=0}^{\infty}nP_{\alpha n}\left( x\right) t^{\alpha \left( n+1\right) }
\end{aligned}
\end{eqnarray}
Comparing the corresponding coefficients of $t^{\alpha n}$ in both sides of \eqref{May} yields
$$x^{\alpha }P_{\alpha n}\left( x\right) -P_{\alpha \left( n-1\right)}\left(
x\right) =\left( n+1\right) P_{\alpha \left( n+1\right)}\left( x\right)
-2nx^{\alpha }P_{n\alpha }\left( x\right) +\left( n-1\right) P_{\alpha \left(
n-1\right)}\left( x\right). $$
Therefore
\begin{equation*}
\left( n+1\right) P_{\alpha \left( n+1\right)}\left( x\right) =\left(2n+1\right) x^{\alpha }P_{\alpha n}\left( x\right) -nP_{\alpha \left( n-1\right)}\left( x\right),
\end{equation*}
which is the required relation.
\end{proof}

Next, concerning differential recurrence relations we deduce the following.
\begin{thm}
The conformable fractional Legendre function $P_{\alpha n}\left( x\right) $
satisfies the following differential recurrence relations:
\begin{align}
&D^{\alpha }P_{\alpha \left( n+1\right) }\left( x\right) -\left( n+1\right)
\alpha P_{\alpha n}\left( x\right) -x^{\alpha }D^{\alpha }P_{\alpha n}\left(
x\right) =0\label{sec6, eq2}\\
&x^{\alpha }D^{\alpha }P_{\alpha n}\left( x\right) -n\alpha P_{\alpha
n}\left( x\right) -D^{\alpha }P_{\alpha \left( n-1\right) }\left( x\right) =0\label{sec6, eq3}\\
&D^{\alpha }P_{\alpha \left( n+1\right) }\left( x\right) -D^{\alpha
}P_{\alpha \left( n-1\right) }\left( x\right) =\left( 2n+1\right) \alpha
P_{\alpha n}\left( x\right) \label{sec6, eq4}
\end{align}
\end{thm}
\begin{proof}
According to the fractional differential operator and in view of   \eqref{sec3, eq1}, it follows that 
\begin{align*}
D^{\alpha }\left( 1-2x^{\alpha }t^{\alpha
}+t^{2\alpha }\right) ^{\frac{-1}{2}}&=D^{\alpha }
\sum\limits_{n=0}^{\infty
}P_{\alpha n}\left( x\right)t^{\alpha n}\\
\alpha t^{\alpha }\left( 1-2x^{\alpha}t^{\alpha}+t^{2\alpha }
\right) ^{\frac{-1}{2}}&=\left( 1-2x^{\alpha}t^{\alpha}+t^{2\alpha }\right)
\sum\limits_{n=0}^{\infty }D^{\alpha}P_{\alpha n}\left( x\right) t^{\alpha n}.
\end{align*}
Also, using  \eqref{sec3, eq1}, it is easily verify 
\begin{equation}\label{equatingnow}
\alpha \sum\limits_{n=0}^{\infty }P_{\alpha n}\left(x\right)t^{\alpha \left( n+1\right)}=\sum\limits_{n=0}^{\infty}
D^{\alpha}P_{\alpha n}\left( x\right) t^{\alpha n}-\sum\limits_{n=0}^{\infty}2x^{\alpha }
D^{\alpha }P_{\alpha n}\left(x\right) t^{\alpha \left( n+1\right)}+\sum\limits_{n=0}^{\infty }
D^{\alpha}P_{\alpha n}\left( x\right)t^{\alpha \left( n+2\right)}.
\end{equation}
Equating the coefficients of $t^{\alpha \left( n+1\right)}$ in both sides of \eqref{equatingnow}, we get  
\begin{equation}\label{sec6, eq5}
\alpha P_{\alpha n}\left( x\right) =D^{\alpha}P_{\alpha \left(n+1\right)}\left( x\right)-2x^{\alpha }
D^{\alpha}P_{\alpha n}\left( x\right) +D^{\alpha }P_{\alpha \left( n-1\right)}\left(x\right). 
\end{equation}
Again differentiate \eqref{sec6, eq1} with respect to $x$ in the sense of conformable fractional we deduce
\begin{equation}\label{sec6, eq6}
D^{\alpha }P_{\alpha \left(
n+1\right)}\left( x\right) =\frac{\left(2n+1\right) }{n+1}
\alpha P_{\alpha n}\left( x\right) +\frac{\left(2n+1\right) }{n+1}x^{\alpha
}D^{\alpha }P_{\alpha n}
\left( x\right) -\frac{n}{n+1}D^{\alpha }P_{\alpha \left(
n-1\right)}\left( x\right). 
\end{equation}
Hence using \eqref{sec6, eq5} and \eqref{sec6, eq6} we obtain \eqref{sec6, eq2} and \eqref{sec6, eq3} respectively and
the combination of \eqref{sec6, eq2} and \eqref{sec6, eq3} gives \eqref{sec6, eq4}. 
\end{proof}
\section{Conformable fractional integral of CFLPs}\label{Sec6}
Along with definition \ref{DefFI}, it is easy to get the following.

\begin{cor}\label{cor6.1} \cite{khalil2014new}
Suppose that $f:\left[ 0,\infty \right) \rightarrow  \mathbb{R},$ then for $\gamma \in (0,1]$, the conformable fractional integral $I_{\gamma }f\left( x\right)$ of order $\gamma $ of $f$ is given by  
\begin{equation}\label{Eq 6.1}
 I_{\gamma }f\left( x\right)
=\int\limits_{0}^{x}t^{\gamma -1}f\left( t\right) dt
\end{equation}
\end{cor}

\begin{lem}\label{sec6, lem1} \cite{khalil2014new}
Suppose that $f:\left[ 0,\infty \right) \rightarrow \mathbb{R}$ is $\gamma $-differentiable for $\gamma \in (0,1],$ then for all $x>0$ one can write: 
\begin{equation}\label{6.2}
I_{\gamma }D^{\gamma }\left( f\left( x\right) \right)
=f\left(x\right) -f\left( 0\right) 
\end{equation}
\end{lem}
With the aid of lemma \ref{sec6, lem1} and corollary \ref{cor6.1},  the following result can be deduced.
\begin{thm}
For $\gamma \in (0,1]$, then the conformable fractional integral $I_{\gamma }$
of $\alpha$-CFLPs can be written as 
\begin{equation}\label{sec7, eq3}
I_{\gamma }P_{\alpha n}\left( x\right) =\sum\limits_{k=0}^{\left\lfloor \frac{n}{2}\right\rfloor}
\frac{\left(-1\right) ^{k}\left( 2n-2k\right) !}{2^{n}k!\left( n-k\right)
!\left( n-2k\right) !
\left( \left( n-2k\right)\alpha +\gamma \right)}
x^{\alpha \left( n-2k\right) +\gamma}
\end{equation}
\end{thm}
\begin{proof}
In view of \eqref{Eq 6.1} and  \eqref{6.2}, we obtain 
\begin{equation*}
\begin{split}
I_{\gamma }P_{\alpha n}\left(x\right) &=\int\limits_{0}^{x}t^{\gamma -1}
\sum\limits_{k=0}^{\left\lfloor \frac{n}{2}\right\rfloor }
\frac{\left( -1\right) ^{k}\left( 2n-2k\right) !}{2^{n}k!\left( n-k\right) !\left( n-2k\right) !}t^{\alpha \left( n-2k\right)}dt \\
&=\sum\limits_{k=0}^{\left\lfloor \frac{n}{2}\right\rfloor }
\frac{\left(-1\right) ^{k}\left( 2n-2k\right) !}{2^{n}k!\left( n-k\right)!\left(n-2k\right) !}\int\limits_{0}^{x}t^{ \alpha \left( n-2k\right)+\gamma-1}dt \\
&=\sum\limits_{k=0}^{\left\lfloor \frac{n}{2}\right\rfloor }
\frac{\left(-1\right) ^{k}\left( 2n-2k\right) !}{2^{n}k!\left( n-k\right)!\left(n-2k\right) !\left( \left( n-2k\right) \alpha +\gamma \right) }x^{ \alpha \left(n-2k\right)+\gamma }
 \end{split}
\end{equation*}
and the result follows. 
\end{proof}
\begin{rem}
If $\gamma = \alpha $ in \eqref{sec7, eq3} we have

 \begin{equation*}
I_{\alpha }P_{\alpha n}\left( x\right)=\sum\limits_{k=0}^{\left\lfloor 
\frac{n}{2}\right\rfloor }\frac{\left(-1\right) ^{k}\left( 2n-2k\right) !}{2^{n}k!\left( n-k\right) !
\left(n-2k\right) !\left( n-2k+1\right) \alpha }x^{ \alpha \left( n-2k+1\right)}
\end{equation*}
\end{rem}
\section{ Orthogonality relation and its applications}\label{Sec7}
It is worth mentioning that definition \ref{def 2.1}   opens the door to refine the orthogonality relation  for the case  $m=n$.  We introduce the following interesting result which will be useful in the applications.
\begin{thm}\label{sec7, thm2}
The conformable  fractional Legendre polynomial satisfies
\begin{equation}\label{sec8, eq1}
\int\limits_{-1}^{1}\left[ P_{\alpha n}\left(x\right) \right] ^{2} d_{\alpha }x=\frac{2}{\alpha \left(2n+1\right) },
\end{equation}
where $d_{\alpha }x=x^{\alpha -1}~dx$.
\end{thm}
\begin{proof}
Mindful of the relation \eqref{sec3, eq1} we get
\begin{equation*}
\frac{1}{1-2x^{\alpha }t^{\alpha }+t^{2\alpha }}=\sum\limits_{n=0}^{\infty }\left[ P_{\alpha n}\left( x\right) \right]^{2}t^{2\alpha n }
\end{equation*}
By applying the conformable fractional integral over $[-1,1]$, it follows that
\begin{equation}\label{sec8, eq2}
\sum\limits_{n=0}^{\infty}\int\limits_{-1}^{1}\left[ P_{\alpha n}\left(x\right)
 \right]^{2}t^{2\alpha n}d_{\alpha }x=\int\limits_{-1}^{1}\frac{d_{\alpha }x}{1-2x^{\alpha }t^{\alpha }+t^{2\alpha }}
\end{equation}
The right hand side of \eqref{sec8, eq2} can be written as
\begin{equation}\label{sec8, eq3}
\int\limits_{-1}^{1}\frac{d_{\alpha }x}{1-2x^{\alpha}t^{\alpha}
+t^{2\alpha }}=\int\limits_{-1}^{1}\frac{x^{\alpha -1}~d~x}{1-2x^{\alpha}t^{\alpha }+t^{2\alpha }}=\frac{-1}{2\alpha t^{\alpha }}\left[\ln \left(1-2x^{\alpha }t^{\alpha }+t^{2\alpha }\right) \right] _{-1}^{1}
\end{equation}
Keeping in mind that the fraction $\alpha $ be in the form  $\frac{1}{k}$ for  $k$ an odd natural number. Therefore due to definition \ref{def 2.1} we have 
\begin{equation*}
\begin{split}
\int\limits_{-1}^{1}\frac{d_{\alpha }x}{1-2x^{\alpha}t^{\alpha}+
t^{2\alpha }}&=\frac{-1}{2\alpha t^{\alpha }}\left[ \ln \left(1-2t^{\alpha}+t^{2\alpha }\right) -\ln \left( 1+2t^{\alpha }+t^{2\alpha}\right) \right]\\
&=\frac{1}{\alpha t^{\alpha }} \ln \left(\frac{ 1+t^{\alpha }}{ 1-t^{\alpha }}\right)
\end{split}
\end{equation*}

Hence in view of  \eqref{sec8, eq3}, we obtain  
\begin{equation*}
\begin{split}
\int\limits_{-1}^{1}\frac{d_{\alpha }x}{1-2x^{\alpha}t^{\alpha}+t^{2\alpha }}
&=\frac{1}{\alpha t^{\alpha }}\left[ 2\left(t^{\alpha }+\frac{t^{3\alpha }}{3}+\frac{t^{5\alpha }}{5}+\dots\right) \right] \\
&=\frac{2}{\alpha }
\left[ 1+\frac{t^{2\alpha }}{3}+\frac{t^{3\alpha }}{5}+\dots\right] =\frac{2}{\alpha }\sum\limits_{n=0}^{\infty }
\frac{1}{2n+1}~t^{2\alpha n}
\end{split}
\end{equation*}
Thus, using \eqref{sec8, eq2} the result is therefore established.
\end{proof}

\subsection{Overview of approximation theory}
Along with the scalar case, one can see that a set of conformable fractional polynomials constitutes a simple set of fractional polynomials \footnote{For the definition of the simple set of polynomials we refer to \cite{rainville1969special}.}.
This leads to the following result which is the analog of that one given in \cite{rainville1969special}.
\begin{prop}\label{propsition 7.1}
Let $\left\{ \varphi _{\alpha n}\left( x\right) \right\} $ be a simple set of
fractional polynomials. If $P\left( x\right) $ is a polynomial of degree $ \alpha m$, there exist constants $a_{k}$ such that 
\begin{equation}\label{sec8, eq5}
P\left( x\right) =\sum\limits_{k=0}^{m}a_{k}\varphi _{\alpha k}\left(x\right), 
\end{equation}
where $a_{k}$ are functions of $k$ and of any parameters involved in $P\left(x\right) $.
\end{prop}

\subsubsection{An expansion theorem }
For the purpose of our study, in the forthcoming subsection, we seek an expansion in
terms of CFLPs of the form
\begin{equation}\label{sec8, eq6}
f\left( x\right) =\sum\limits_{n=0}^{\infty }a_{n}P_{\alpha n}\left( x\right) ,~\ \ \left\vert x\right\vert <1,~\ \ \alpha \in (0,1].
\end{equation}
Note that the series on the right-hand side actually convergences to $f\left(x\right) $ providing that $f\left(x\right) $ is sufficiently well behaved.
To obtain such a formula \eqref{sec8, eq6} we first establish the expansion of $x^{\alpha n},~\ \alpha \in (0,1]$ in a series of CFLPs. The previous proposition \ref{propsition 7.1} tells
us that any polynomial can be expanded in a series of conformable fractional Legendre polynomials merely because the $P_{\alpha n}\left( x\right) $ form a simple set. 
The orthogonality of the set $\left\{ P_{\alpha n}\left( x\right) \right\} $ plays a vital role only in the determination of coefficients. So that the expansion of $x^{\alpha n}$ in a series of CFLPs is useful.
\subsubsection{The expansion of $x^{\alpha n},\alpha \in (0,1]$}
\begin{thm}\label{sec8, thm3}
For non-negative integer $n$ and for $\alpha \in (0,1]$ we have the expansion of $x^{\alpha n}$ in terms of CFLPs in the form   
\begin{equation}\label{sec8, eq7}
x^{\alpha n}=\frac{n!}{2^{n}}\sum\limits_{k=0}^{\left\lfloor \frac{n}{2} \right\rfloor }\frac{\left( 2n-4k+1\right) }{k!\left( \frac{3}{2}\right) _{n-k}}P_{\alpha \left( n-2k\right) }\left( x\right) 
\end{equation}
\end{thm}
\begin{proof}
Going back to \eqref{sec3, eq1}  we get 
\begin{equation*}
\begin{split} 
\sum\limits_{n=0}^{\infty }P_{n\alpha }\left( x\right) t^{n\alpha }&=\left( 1-2x^{\alpha }t^{\alpha }+t^{2\alpha }\right) ^{\frac{-1}{2}}\\
&=\left(1+t^{2\alpha }\right) ^{\frac{-1}{2}}\left( 1-\frac{2x^{\alpha }t^{\alpha }}{1+t^{2\alpha }}\right) ^{\frac{-1}{2}}\\
&=\left( 1+t^{2\alpha }\right) ^{\frac{-1}{2}}~_{1}F_{0}\left( \frac{1}{2};-;\frac{2x^{\alpha }t^{\alpha }}{1+t^{2\alpha }}\right)\\
&=\left( 1+t^{2\alpha }\right) ^{\frac{-1}{2}}\sum\limits_{n=0}^{\infty }\frac{\left( \frac{1}{2}\right) _{n}}{n!}\left( \frac{2x^{\alpha }t^{\alpha }}{1+t^{2\alpha }}\right) ^{n}.
\end{split}
\end{equation*} 
Therefore
\begin{equation}\label{sec8, eq8}
\sum\limits_{n=0}^{\infty }\frac{\left( \frac{1}{2}\right) _{n}\left(
2x^{\alpha }\right) ^{n}}{n!}\left( \frac{t^{\alpha }}{1+t^{2\alpha }}
\right) ^{n}=\left( 1+t^{2\alpha }\right) ^{\frac{1}{2}}\sum\limits_{n=0}^{
\infty }P_{n\alpha }\left( x\right) t^{n\alpha }
\end{equation}
Putting $t^{\alpha }=\frac{2u^{\alpha }}{1+\sqrt{1-4u^{2\alpha }}},$ it easily follows that \newline
\begin{equation*}
1+t^{2\alpha }=\frac{2}{1+\sqrt{1-4u^{2\alpha }}}~ \text{and then}~ u^{\alpha }=\frac{
t^{\alpha }}{1+t^{2\alpha }}.
\end{equation*}
Hence, equation \eqref{sec8, eq8} verifies 
\begin{equation}\label{sec8, eq9}
\sum\limits_{n=0}^{\infty }\frac{\left( \frac{1}{2}\right) _{n}\left(
2x^{\alpha }\right) ^{n}}{n!}u^{\alpha n}=\sum\limits_{n=0}^{\infty
}P_{n\alpha }\left( x\right) u^{n\alpha }\left[ \frac{2}{1+\sqrt{
1-4u^{2\alpha }}}\right] ^{n+\frac{1}{2}}
\end{equation}
Therefore we have
\begin{equation*}
\begin{split} 
\left[ \frac{2}{1+\sqrt{1-4u^{2\alpha }}}\right] ^{n+\frac{1}{2}}&=~_{2}F_{1}\left[ \frac{1}{2}
\left( n+\frac{1}{2}\right) ,\frac{1}{2}\left(n+\frac{3}{2}\right) ;\left( n+\frac{3}{2}\right) ;4u^{2\alpha }\right]\\
&=\sum\limits_{k=0}^{\infty }\frac{\left( \frac{1}{2}\left( n+\frac{1}{2}\right) 
\right) _{k}\left( \frac{1}{2}\left( n+\frac{3}{2}\right) \right)_{k}}{\left( n+\frac{3}{2}\right) _{k}k!}2^{2k}u^{2\alpha k}\\
&=\sum\limits_{k=0}^{\infty }\frac{\left( n+\frac{1}{2}\right) _{2k}}{\left(n+\frac{3}{2}\right) _{k}k!}u^{2\alpha k}\\
&=\sum\limits_{k=0}^{\infty }\frac{\left( 2n+1\right) \left( \frac{1}{2}\right) _{n+2k}}{k!\left( \frac{3}{2}\right) _{n+k}}u^{2\alpha k}.
\end{split}
\end{equation*} 
Then in view of  \eqref{sec8, eq9} we obtain
\begin{equation}\label{Eq7.9}
\begin{split} 
\sum\limits_{n=0}^{\infty }\frac{\left( \frac{1}{2}\right) _{n}\left(2x^{\alpha }\right) ^{n}}{n!}u^{\alpha n}&=
\sum\limits_{n=0}^{\infty}\sum\limits_{k=0}^{\infty }\frac{\left( 2n+1\right) \left( \frac{1}{2}
\right) _{n+2k}}{k!\left( \frac{3}{2}\right) _{n+k}}P_{n\alpha }\left(x\right) u^{\left( n+2k\right)
 \alpha }\\
&=\sum\limits_{n=0}^{\infty}\sum\limits_{k=0}^{\left\lfloor \frac{n}{2}\right\rfloor }
\frac{\left(2n-4k+1\right) \left( \frac{1}{2}\right) _{n}}{k!\left( \frac{3}{2}\right)_{n-k}}P_{\left( n-2k\right) \alpha }
\left( x\right) u^{n\alpha }.
\end{split}
\end{equation}
Equating the coefficients of $u^{\alpha n}$ in both sides of \eqref{Eq7.9}, we get  
\begin{equation*}
\frac{\left( \frac{1}{2}\right) _{n}\left( 2x^{\alpha }\right) ^{n}}{n!}
=\sum\limits_{k=0}^{\left\lfloor \frac{n}{2}\right\rfloor }\frac{\left(2n-4k+1\right) 
\left( \frac{1}{2}\right) _{n}}{k!\left( \frac{3}{2}\right)_{n-k}}P_{\left( n-2k\right) \alpha }\left( x\right) 
\end{equation*}
Therefore,
\begin{equation*}
x^{n\alpha }=\frac{n!}{2^{n}}\sum\limits_{k=0}^{\left\lfloor \frac{n}{2}
\right\rfloor }\frac{\left( 2n-4k+1\right) }{k!\left( \frac{3}{2}\right)_{n-k}}P_{\left( n-2k\right) \alpha }\left( x\right). 
\end{equation*}
\end{proof}
\subsubsection{The expansion of analytic functions}
\noindent
Theorem \ref{sec8, thm3} can be employed to get  explicit formulas for the coefficients in the expansion of analytic functions by means of  CFLPs series. 
Such type of expansion theory was treated classically in several  approaches, we may mention for example \cite{whittaker1927modern, szeg1939orthogonal}. For a general study of the theory of expansion  of analytic function in terms of series of polynomials, we refer to the works of Whittaker \cite{whittaker1949sries}, Boas \cite{boas1958polynomial} and in higher dimensions (Clifford analysis) by Abul-Ez et al. \cite{abul1990basic, zayed2012generalized, zayed2020generalized}.

In usual classical  calculus, Taylor's power series representation of a function $f$ around certain points is not always guaranteed, unlike the case in the theory of conformable fractional calculus. 
Abdeljawad \cite{abdeljawad2015conformable} introduced the fractional power series expansion for an infinity $\alpha -$differentiable functions via the following fundamental result.
\begin{thm}\label{sec8, thm4}
Let $f$  be an infinitely  conformable fractional differentiable function for some $0<\alpha \leq 1$ around a point $x_{0}.$ Then $f$ admits a Taylor conformable fractional power series expansion as:
\begin{equation}\label{sec8, eq10}
f\left( x\right) =\sum\limits_{k=0}^{\infty }\frac{\left( D^{\alpha
}f\right) ^{\left( k\right) }\left( x_{0}\right)}{\alpha ^{k}k!}\left( x-x_{0}\right) ^{\alpha k},~~~x_{0}<x<x_{0}+R^{1/\alpha },~~\ R>0
\end{equation}
\end{thm}
Here $\left( D^{\alpha }f\right) ^{\left( k\right) }\left(
x_{0}\right) $ means that repeated application of the conformable fractional
derivative $k$ times at a point $x=x_{0}.$

Taking $x_{0}=0$ one easily deduces the analog of $\alpha $-Maclaurin expansion in the form 
\begin{equation}\label{sec8, eq11}
f\left( x\right) =\sum\limits_{k=0}^{\infty }\frac{\left( D^{\alpha}f\right) ^{\left( k\right) }\left( 0\right) }{\alpha ^{k}k!} x^{\alpha k}
,~~~0<x<R^{1/\alpha },~~\ R>0,
\end{equation}
where also $\left( D^{\alpha }f\right) ^{\left( k\right) }\left( 0\right) 
$ means that repeated application of the conformable fractional derivative $k$ times at a point $x=0.$ \newline
With the aid of the formula \eqref{sec8, eq11} and applying theorem \eqref{sec8, thm3} attain the following. \newline 
\begin{equation*}
f\left( x \right) =\sum\limits_{k=0}^{\infty}\sum\limits_{k=0}^{\left\lfloor \frac{n}{2}\right\rfloor }
\left[ \frac{a_{n}}{n!}\frac{n!}{2^{n}}\frac{\left( 2n-4k+1\right) }{k!\left( \frac{3}{2}\right) _{n-k}}P_{\left( n-2k\right) 
\alpha }\left( x\right) \right]. 
\end{equation*}
Hence, 
\begin{equation}\label{sec8, eq12}
f\left( x\right) =\sum\limits_{k=0}^{\infty }\sum\limits_{k=0}^{\infty }\frac{\left( 2n+1\right) a_{n+2k}}{2^{n+2k}k!\left( \frac{3}{2}\right) _{n+k}}P_{n\alpha }\left( x\right).
\end{equation}
\section{Shifted conformable fractional Legendre polynomials (SCFLPs)}\label{Sec8}
Abu Hammed and Khalil \cite{hammad2014legendre} considered the following conformable fractional Legendre differential equation 
\begin{equation}\label{sec11, eq1}
\left( 1-t^{2\alpha }\right) D^{\alpha }D^{\alpha }P_{\alpha n}\left(t\right) 
-2t^{\alpha }D^{\alpha }P_{\alpha n}\left( t\right) +n\left(n+1\right) P_{\alpha n}\left( t\right) =0.
\end{equation}
The transformation $t^{\alpha }=2x^{\alpha }-1$ gives the shifted conformable fractional differential equation  equations in the form.
\begin{equation}\label{sec11, eq2}
x^{\alpha }\left( 1-x^{\alpha }\right) D^{\alpha }D^{\alpha }P_{\alpha n}^{\ast}
\left( x\right)-\left( 2x^{\alpha }-1\right) D^{\alpha }P_{\alpha n}^{\ast }\left( x\right) 
+n\left( n+1\right) P_{\alpha n}^{\ast }\left(x\right) =0.
\end{equation}
Note that $P_{\alpha n}\left( t\right)$ and  $P_{\alpha n}^{\ast }\left( x\right) $
represent  the CFLPs and SCFLPs, respectively. 
Accordingly the shifted fractional Legendre polynomials of degree $n\alpha ,$ can be stated by means of recurrence relation \eqref{sec6, eq1} as
\begin{equation}\label{sec11, eq3}
P_{\alpha \left( n+1\right) }^{\ast }\left( x\right) =\frac{\left(2n+1\right) \left( 2x^{\alpha }-1\right) }{n+1}P_{\alpha n}
^{\ast }\left(x\right) -\frac{n}{n+1}P_{\alpha \left( n-1\right) }^{\ast }\left( x\right), n=1,2,\dots,
\end{equation}
where $P_{0}^{\ast }\left( x\right) =1$ and $ P_{\alpha }^{\ast }\left( x\right)=\left( 2x^{\alpha }-1\right).$ \newline
The analytical formula of SCFLPs  of degree $\alpha n$ will be such that 
\begin{equation}\label{sec11, eq4}
P_{\alpha n}^{\ast }\left( x\right) =\sum\limits_{k=0}^{n}\frac{\left(-1\right) 
^{n+k}\left( n+k\right) !}{\left( n-k\right) !\left( k!\right) ^{2}}x^{\alpha k}.
\end{equation}
As shown previously in theorem \ref{sec7, thm2}, the SCFLPs are orthogonal  over $[0,1],$ hence we may write

\begin{equation}
\int\limits_{0}^{1}P_{\alpha n}^{\ast }\left( x\right) P_{\alpha m}^{\ast}\left( x\right) x^{\alpha -1}dx
=\frac{1}{\alpha \left( 2n+1\right) }\delta_{nm},
\end{equation}
where $\delta _{nm}$ is the Kronker delta function. 

Now, we are going to establish the following famous formula for the SCFLPs. 
\subsection{Rodrigues formula for SCFLPs}
It is well known that one of the basic ways to define a sequence of orthogonal polynomials is to
use their Rodrigues' formula \cite{rusev2005classical}. If it is known, then a lot of nice properties of the
 polynomials can be derived. That is why generalizations of these formulas occupy mathematicians’
attention in last two decades, both to define new classes of special functions and polynomials
 and to include fractional order differentiation. Starting from a Rodrigues formula for shifted
Legendre polynomials, Rajkovi\'{c} and Kiryakova \cite{rajkovic2010legendre} investigated the special functions thus they
defined. Further they studied the orthogonality property  which held only for some
special cases.

According to the notation of CFD we have $D^{\alpha n}=D^{\alpha}.D^{\alpha}.D^{\alpha}\dots D^{\alpha},n-$ times,  and due to the fact $D^{\alpha } x^{p}=px^{p-\alpha },$ the authors  in \cite{hammad2014legendre} stated the  following  Rodrigues formula of the CFLPs $P_{\alpha n}\left( x\right)$ in the form 
\begin{equation}\label{cflpsRedrigues}
P_{\alpha n}\left( x\right) =\frac{1}{\alpha ^{n}.2 ^{n}n!}D^{\alpha n} \left( x^{2\alpha }-1\right) ^{n}.
\end{equation}
Along with the formula \eqref{cflpsRedrigues}, the following theorem provides  the Rodrigues formula for the SCFLPs, $P_{\alpha n}^{\ast }\left( x\right)$.
\begin{thm}
The shifted conformable fractional Legendre polynomials can be written  in the sense of  conformable derivative as:
\begin{equation}\label{sec11, eq5}
P_{\alpha n}^{\ast }\left( x\right) =\frac{1}{\alpha ^{n}n!}D^{\alpha n}\left[ x^{\alpha n}\left( x^{\alpha }-1\right) ^{n}\right]. 
\end{equation}
\end{thm}
\begin{proof}
First observe that 
\begin{equation*}
x^{\alpha n}\left( x^{\alpha }-1\right) ^{n}=\sum\limits_{i=0}^{n}\frac{n!}{\left( n-i\right)!i!}\left(x^{2\alpha }\right) ^{n-i}\left( -x^{\alpha}\right)^{i}
=\sum\limits_{i=0}^{n}\frac{\left( -1\right) ^{i}~\ n!}{\left(n-i\right) !i!}x^{\alpha \left( 2n-i\right). }
\end{equation*}
Thus, in virtue of  the CFD definition we have
\begin{eqnarray*}
\begin{aligned}
 D^{\alpha n}\left[ x^{\alpha n}\left( x^{\alpha }-1\right)
^{n}
\right] &=\sum\limits_{i=0}^{n}\frac{\left( -1\right) ^{i}~n!}{
\left(
n-i\right) !i!}D^{\alpha n}\left[ x^{\alpha \left( 2n-i\right) }
\right] \\
&=\alpha ^{n}\left( n!\right)
\sum\limits_{i=0}^{n}\frac{\left(
-1\right) ^{i}~\ \left( 2n-i\right) !}{
\left[ \left( n-i\right) !\right]
^{2}i!}x^{\alpha \left( n-i\right) }.
\end{aligned}
\end{eqnarray*}
Putting $k=n-i,$ thus yielding
\begin{eqnarray*}
D^{\alpha n}\left[ x^{\alpha n}\left(x^{\alpha }-1\right) ^{n}\right]=\alpha ^{n}\left( n!\right)
\sum\limits_{i=0}^{n}\frac{\left( -1\right)^{n+k}~\ \left( n+k\right) !}{\left[ k!\right] ^{2}\left( n-k\right) !}
x^{\alpha k}=\alpha ^{n}\left(n!\right) {P_{\alpha n}^{\ast}\left(x\right) }.
\end{eqnarray*}
\end{proof}
\subsection{Applications}
It is well known that orthogonal functions have a vital role in dealing with diverse mathematical and physical problems.
In this subsection, we derive a new approximate formula of conformable fractional derivative based on SCFLPs. The known spectral Legendre collocation method is established for solving some interesting FDEs.     
 
The idea behind the used approach (either described in the Caputo sense or others) with Tau or collocation methods (see \cite{saadatmandi2010new, kazem2013fractional, ccerdik2020numerical}), is to derive a general formulation for 
fractional Legendre functions and product operational matrices. These matrices together with either Tau or collocation methods are then utilized to simplify the solution of the proposed problem to the solution of a system of algebraic type  equations.  The proposed technique is described as follows.

In view of relation \eqref{sec11, eq4}, the explicit analytical form of SCFLPs  of degree 
$i\alpha,\alpha \in (0,1]$ may be written as:

\begin{equation}\label{sec12, eq1}
P_{\alpha i}^{\ast }\left( x\right) =\sum\limits_{s=0}^{i}b_{s,i}x^{s\alpha},~~ \text{where} ~~ b_{s,i}=\frac{
\left( -1\right) ^{s+i}\left( i+s\right) !}{\left(i-s\right) !\left(
s!\right) ^{2}}.
\end{equation}
A combination of relation \eqref{sec2, eq4} and theorem \ref{sec7, thm2} we have the orthogonality of the  SCFLPs on $[0,1]$ and therefore: 
\begin{equation}\label{sec12, eq2}
\int\limits_{0}^{1}P_{\alpha
i}^{\ast }\left( x\right) P_{\alpha j}^{\ast }\left( x\right)x^{\alpha -1}dx=\frac{1}{\alpha
\left( 2i+1\right) }\delta _{ij}.
\end{equation}
Assuming  $y=f\left(x\right) $ defined over the interval $[0,1]$ has an expansion in the form
\begin{equation}\label{sec12, eq3}
f\left( x\right) =\sum\limits_{i=0}^{\infty
}a_{i}P_{\alpha i}^{\ast }\left( x\right),
\end{equation}
where $a_{i}$ are determined by:
\begin{equation}\label{sec12, eq4}
a_{i}=\alpha \left( 2i+1\right)
\int\limits_{0}^{1}P_{\alpha i}^{\ast }\left(x\right) f\left( x\right) x^{\alpha
-1}dx, i=0,1,2,\dots
\end{equation}
In practice, only the first $\left(m+1\right) $ terms of SCFLPs are considered, so that
\begin{equation}\label{sec12, eq5}
y_{m}\left( x\right)
=\sum\limits_{i=0}^{m}a_{i}P_{\alpha i}^{\ast }\left( x\right). 
\end{equation}

\begin{thm}\label{sec12, thm1}
The conformable fractional derivative of the SCFLPs of order $\gamma >0$ can be formulated 
as:
\begin{equation}\label{sec12, eq6}
D^{\gamma }P_{\alpha i}^{\ast }\left( x\right)
=\sum\limits_{s=0}^{i}b_{s,i}^{\prime
}\frac{\Gamma \left( \alpha
s+1\right) }{\Gamma \left( \alpha s-\left\lfloor
\gamma \right\rfloor
\right) }x^{s\alpha -\gamma },
\end{equation}
where $b_{s,i}^{\prime }=0$
when \bigskip $s\alpha \in \mathbb{N}_{0}$ and $s\alpha <\gamma,$ on the
otherwise $b_{s,i}^{\prime }=b_{s,i}$
\end{thm}

\begin{proof}

From the linearity of the conformable derivative see \cite{ abdeljawad2015conformable, khalil2014new}, we have: 
\begin{equation*}
D^{\gamma }P_{\alpha i}^{\ast }\left( x\right)
=\sum\limits_{s=0}^{i}b_{s,i}^{~}D^{\gamma }x^{\alpha
s}=\sum\limits_{s=0}^{i}b_{s,i}
^{\prime }\frac{\Gamma \left( \alpha
s+1\right) }{\Gamma \left( \alpha s-\left\lfloor \gamma \right\rfloor
\right) }x^{\alpha s-\gamma },
\end{equation*}
where $b_{s,i}^{\prime }=0$ when \bigskip $s\alpha \in \mathbb{N}_{0}$
and $s\alpha <\gamma$ on the otherwise $b_{s,i}^{\prime }=b_{s,i}$.
\end{proof}
\begin{thm}\label{sec12, thm2}
Suppose that $y_{m}\left( x\right) $ is an approximated function
given by means of  shifted conformable fractional Legendre polynomials \eqref{sec12, eq5}, then
we have:
\begin{equation}\label{sec12, eq7}
D^{\gamma }y_{m}\left(
x\right)=\sum\limits_{i=0}^{m}\sum\limits_{s=0}^{i}a_{i}R_{i,s}^{\left(
\gamma \right) }x^{s\alpha -\gamma },
\end{equation}
where $R_{i,s}^{\left(
\gamma \right) }=b_{s,i}^{\prime }\frac{\Gamma \left(\alpha s+1\right) }{
\Gamma \left( \alpha s-\left\lfloor \gamma \right\rfloor \right) }.$
\end{thm}
\begin{proof}
In virtue of the linearity property of the CFD and using the result of theorem \ref{sec12, thm1}, the proof is therefore completed.
\end{proof}
Now, suppose that  the generalized linear multi-order conformable fractional
differential equation is such that: 
\begin{equation}\label{sec12, eq8}
D^{\gamma }y\left(x\right) +\sum\limits_{r=1}^{k}A_{r}D^{\gamma_{r}}y\left( x\right)+A_{k+1}y\left( x\right) 
=A_{k+2}g\left( x\right),\ x\in \left[ 0,1\right] 
\end{equation}
with the initial conditions
\begin{equation}\label{sec12, eq9}
y^{\left(
i\right) }\left( 0\right) =d_{i},i=0,1,2,\dots,\left\lceil \gamma \right\rceil-1,
\end{equation}
where $0<\gamma _{1}<\gamma _{2}<\dots< \gamma
_{k}<\gamma,$ $ D^{\gamma }$ refers to the conformable 
fractional derivative of order $\gamma$,  $g\left(x\right) $ are known to be continuous functions
 and $d_{i},i=0,1,2,\dots,\left\lceil \gamma \right\rceil -1$ are some constants. 
 
Suppose that  the conformable fraction differential equation \eqref{sec12, eq8} has a solution in the form: 
\begin{equation}\label{sec12, eq10}
y_{m}\left( x\right)
=\sum\limits_{i=0}^{m}a_{i}P_{\alpha i}^{\ast }\left( x\right) 
\end{equation}
It can be easily deduced from  \eqref{sec12, eq8}, \eqref{sec12, eq10} and the theorem \ref{sec12, thm2} that 
\begin{equation}\label{sec12, eq11}
\sum\limits_{i=0}^{m}\sum\limits_{s=0}^{i}a_{i}R_{i,s}^{\left(
\gamma \right) }
x^{s\alpha -\gamma
}+\sum\limits_{r=1}^{k}A_{r}\left\{\sum\limits_{i=0}^{m}\sum\limits_{s=0}^{i}a_{i}R_{i,s}
^{\left( \gamma_{r}\right) }x^{s\alpha
-\gamma _{r}}\right\}+A_{k+1}\sum\limits_{i=0}^{m}a_{i}P_{\alpha i}^{\ast }\left(
x\right)=A_{k+2}g\left( x\right) 
\end{equation}
Collocating Equation \eqref{sec12, eq11} at points $x_{p},p=1,2,\dots,\left(m+1-\left\lceil \gamma \right\rceil \right),$ we have: 
\begin{equation}\label{sec12, eq12}
\sum\limits_{i=0}^{m}\sum\limits_{s=0}^{i}a_{i}R_{i,s}^{\left( \gamma \right) }x_{p}^{s\alpha -\gamma
}+\sum\limits_{r=1}^{k}A_{r}\left\{\sum\limits_{i=0}^{m}\sum\limits_{s=0}^{i}a_{i}R_{i,s}
^{\left( \gamma_{r}\right) }x_{p}^{s\alpha
-\gamma _{r}}\right\}+A_{k+1}\sum\limits_{i=0}^{m}a_{i}P_{\alpha i}^{\ast }\left(
x_{p}\right)=A_{k+2}g\left( x_{p}\right) 
\end{equation}
Using the roots of shifted conformable fractional Legendre polynomials
 $P_{\alpha \left( m+1-\left\lceil \gamma \right\rceil
\right) },$ we get suitable collocation points. Moreover by employing \eqref{sec12, eq10} into the initial conditions \eqref{sec12, eq9}, we obtain
 $\left\lceil \gamma \right\rceil $-equations.  
This gives a linear algebraic system  consisting of $\left( m+1\right) $ equations in the unknowns $
a_{i},i=0,1,2,...,m.$ 
 By solving this system,  the solution of the initial value problem \eqref{sec12, eq8} can be determined. 
 
The above method can be employed to solve some interesting  fractional differential equations. For examples we have 
\begin{ex}[Bagley-Torvik equation \cite{saadatmandi2010new}]
The inhomogeneous Bagley-Torvik initial value problem is
\begin{equation}\label{sec12, eq13}
D^{2}y\left( x\right) +D^{3/2}y\left(
x\right) +y\left( x\right) =1+x,~~ x\in  [0,1]
\end{equation}
with the initial assumption 
\begin{equation}\label{sec12, eq14}
y\left( 0\right) =1,   y^{\prime
}\left( 0\right) =1
\end{equation}
\end{ex}
The conformable fractional differential equation \eqref{sec12, eq13} has an exact solution as $ y\left( x\right) =1+x.$  
For $m=2$ and $\alpha =1,$ we derive the approximated
analytical solution of \eqref{sec12, eq13} as follows: 
\begin{equation}\label{sec12, eq15}
y_{2}\left( x\right) =\sum\limits_{i=0}^{2}a_{i}P_{i}^{\ast }\left( x\right) 
\end{equation}
Inserting the initial conditions into Equation \eqref{sec12, eq15} we have:
\begin{equation}\label{sec12, eq16}
a_{0}-a_{1}+a_{2}=1,   2a_{1}-6a_{2}=1
\end{equation}
For the collocation point $x_{1}=0.5$ which is the root of the
SCFLPs, for $\alpha =1$,  equation \eqref{sec12, eq11} can be written as
\begin{equation}\label{sec12, eq17}
a_{0}+19.985281374238571a_{2}=1.5
\end{equation}
From \eqref{sec12, eq16} and \eqref{sec12, eq17} we have: 
\begin{equation}\label{sec12, eq18}
a_{0}=3/2,~\ a_{1}=1/2,~\ a_{2}=0
\end{equation}
Therefore, the obtained solution is $y\left( x\right)
=a_{0}P_{0}^{\ast }\left( x\right) +a_{1}P_{1}^{\ast }\left( x\right) +a_{2}P_{2}^{\ast }\left(
x\right) =1+x$ which
represents the exact solution of  \eqref{sec12, eq13}.

\begin{ex}[\cite{ccerdik2020numerical}]
Suppose that we have the following initial value problem with variable
coefficients 
\begin{equation}\label{sec12, eq19}
D^{3/2}y\left( x\right) +2y^{\prime
}\left( x\right) +3\sqrt{x} D^{1/2}y\left( x\right) +\left( 1-x\right)
y\left( x\right) =2\sqrt{x}+4x+7x^{2}-x^{3},~\ \ x\in \lbrack 0,1]
\end{equation}
with the primary  conditions 
\begin{equation}\label{sec12, eq20}
y\left( 0\right)=0,  y^{\prime }\left( 0\right) =0
\end{equation}
\end{ex}
Equation \eqref{sec12, eq19} has an exact solution in the form  $y\left( x\right) =x^{2}$ \newline
For $m=2$ and $\alpha =1,$ we obtain the approximated analytical solution of \eqref{sec12, eq19} as follows: 
\begin{equation}\label{sec12, eq21}
y_{2}\left( x\right)
=\sum\limits_{i=0}^{2}a_{i}P_{i}^{\ast }\left( x\right) 
\end{equation}
Substituting the initial conditions into Equation \eqref{sec12, eq21} we have:
\begin{equation}\label{sec12, eq22}
a_{0}-a_{1}+a_{2}=0,   2a_{1}-6a_{2}=0
\end{equation}
For the collocation point $x_{1}=0.5$ which represent the root of the SCFLPs, $P_{1}^{\ast }\left( x\right) $ and for $\alpha =1$,  equation \eqref{sec12, eq11} can be
written as
\begin{equation}\label{sec12, eq23}
\frac{1}{2}
a_{0}+7a_{1}+8.235281374238571a_{2}=5.039213562373095
\end{equation}
From \eqref{sec12, eq22} and \eqref{sec12, eq23} we have: 
\begin{equation}\label{sec12, eq24}
a_{0}=1/3,  a_{1}=1/2,~\ a_{2}=1/6
\end{equation}
Therefore, we have $y\left( x\right) =a_{0}P_{0}^{\ast }\left(x\right) + a_{1}P_{1}^{\ast }\left( x\right) +a_{2}P_{2}^{\ast }\left( x\right) =x^{2}$
 which is in fact the exact solution of Equation \eqref{sec12, eq19}.

\begin{ex}
Let the  initial value problem with variable
coefficients be such that 
\begin{equation}\label{sec12, eq25}
D^{1/2}y\left( x\right) +\sqrt{x}y\left(
x\right) =1+2x,~\ \ x\in \lbrack 0,1]
\end{equation}
with the initial
condition 
\begin{equation}\label{sec12, eq26}
y\left( 0\right) =0
\end{equation}
\end{ex}
Consequently, we infer that the exact solution in the form $y\left( x\right) =2\sqrt{x}$ \newline
For $m=1$ and $\alpha =1/2,$ we get the approximated analytical solution of
Equation \eqref{sec12, eq25} in the form: 
\begin{equation}\label{sec12, eq27}
y_{1}\left( x\right)
=\sum\limits_{i=0}^{1}a_{i}P_{\left(\frac{1}{2} \right)i}^{\ast }\left( x\right)=a_{0}P_{0}^{\ast }\left( x\right)
+a_{1}P_{\left( \frac{1}{2}\right) }^{\ast }\left(x\right) 
\end{equation}
Substituting the initial condition into  \eqref{sec12, eq27} we have:
\begin{equation}\label{sec12, eq28}
a_{0}-a_{1}=0
\end{equation} 
For the collocation point $
x_{1}=0.25$ which is the root of the shifted conformable fractional Legendre
polynomials $P_{\left( \frac{1}{2}\right)}^{\ast }\left( x\right) =2x^{\left(
1/2\right) }-1$, equation \eqref{sec12, eq11} is given by
\begin{equation}\label{sec12, eq29}
\frac{1}{2}
a_{0}+a_{1}=\frac{3}{2}
\end{equation}
From \eqref{sec12, eq28} and \eqref{sec12, eq29} we have  
$
a_{0}=1,  a_{1}=1$.
Therefore the resulted
solution is $y\left( x\right) =a_{0}P_{0}^{\ast }\left(x\right) +a_{1}P_{\left( 
\frac{1}{2}\right) }^{\ast }\left( x\right) =2\sqrt{x}$
 which is the exact solution of  \eqref{sec12, eq25}.
\begin{rem}
The problem  \eqref{sec12, eq25} is newly presented with more general assumption since we consider $\alpha = \frac{1}{2}$
\end{rem}
\section{Conclusions and closing remarks}\label{Sec9}
There are many definitions of the fractional derivatives. One of the most recent ones is the conformable fractional derivative. Accordingly, the main goal of the this work is to study conformable  fractional  Legendre polynomials which appear in many different areas of mathematics and physics. 
The materials developed in sections 3-7 provide  important properties of these conformable fractional Legendre polynomials.   
First, we laid down the conformable fractional Legendre polynomials via different generating functions. This gave tools to study  their main mathematical and convergence properties. Further, we continue the work initiated by \cite{hammad2014legendre} concerning conformable fractional Legendre polynomials for which we established orthogonality property with related  applications.

 Over the last years, fractional differential equations were focus of intense research due to their importance in the modeling of several physical  phenomena from different areas of science and engineering.

In the current paper, as the SCFLPs and their fundamental properties together with the collocation method are provided, this gives a way to solve some interesting  fractional differential equations subjected to specific initial conditions. Of course, the fractional derivatives are described in the conformable sense. The  truncated Legendre series is taking into account, from which it can be easily  determine the solutions for arbitrary independent variables.  
The validity and applicability of the presented technique  are examined through the given examples. We derive  exact solutions of some examples with finite number of terms. 
We believe that the prescribed scheme can be applied to boundary value problems for CFDEs  and also extended to conformable fractional partial differential equations as well as can be applied to get numerical solutions of the CFDEs via SCFLPs to support our obtained results. It should be observed that the authors in \cite{kazem2013fractional} constructed fractional- order Legendre functions to obtain  solutions of some fractional- order differential equations. Their techniques depending on adapting Caputo's definition by using Riemann-Liouville fractional integral operator as well as the Tau method is involved. In fact their results generalized to fractional setting those  given in \cite{saadatmandi2010new}.

Along with the technique used in \cite{kazem2013fractional}, we generalized those results given in \cite{ccerdik2020numerical} to conformable fractional setting. In fact a conformable fractional differential equations are solved by using collocation method depending  on the SCFLPs. 

It is worth mentioning that while the process of submitting this paper,  the very recent work of A. El-Ajou \cite{el2020modification} just appeared to release a modified construction of conformable fractional calculus and this promises new results. Also, this modification will open the door to find a physical meaning of these modified definitions in various fields of applied sciences. Looking forward for such development in the forthcoming work.

\bigskip    
\noindent \textbf{Acknowledgements:} 
\noindent
The first and third authors would to express their appreciations and thanks to Egyptian ASRT for granting them the project no 6479 via Science UP Faculty of Science Grants 2020.

\end{document}